\RequirePackage{amsthm}

\documentclass[sn-mathphys,Numbered]{sn-jnl}

\usepackage{graphicx}%
\usepackage{multirow}%
\usepackage{amsmath,amssymb,amsfonts}%
\usepackage{amsthm}%
\usepackage{mathrsfs}%
\usepackage[title]{appendix}%
\usepackage{xcolor}%
\usepackage{textcomp}%
\usepackage{manyfoot}%
\usepackage{booktabs}%
\usepackage{algorithm}%
\usepackage{algorithmicx}%
\usepackage{algpseudocode}%
\usepackage{listings}%

\theoremstyle{thmstyleone}%
\newtheorem{theorem}{Theorem}
\theoremstyle{thmstyletwo}%
\newtheorem{example}{Example}%
\newtheorem{lemma}{Lemma}

\theoremstyle{thmstylethree}%

\newcommand{\R}{\mathbb{R}}
\newcommand{\N}{\mathbb{N}}

\newcommand{\tilh}{h_{\Tilde{c}}}
\newcommand{\tilc}{\Tilde{c}}
\newcommand{\tilv}{\Tilde{v}}
\newcommand{\bart}{\Bar{t}}

\newcommand{\barv}{\Bar{v}}
\newcommand{\cmin}{c_{\text{min}}}

\DeclareMathOperator*{\conv}{conv}

\DeclareMathOperator*{\pr}{pr}
\DeclareMathOperator*{\argmin}{arg\,min}

\begin{document}

\title[A note on the convergence of deterministic gradient sampling]{A note on the convergence of deterministic gradient sampling in nonsmooth optimization}


\author*[1]{\fnm{Bennet} \sur{Gebken}}\email{bgebken@math.upb.de}

\affil*[1]{\orgdiv{Insitute of Mathematics}, \orgname{Paderborn University}, \orgaddress{\street{Warburger Str. 100}, \city{Paderborn}, \postcode{33098}, \country{Germany}}}


\abstract{
    Approximation of subdifferentials is one of the main tasks when computing descent directions for nonsmooth optimization problems. In this article, we propose a bisection method for weakly lower semismooth functions which is able to compute new subgradients that improve a given approximation in case a direction with insufficient descent was computed. Combined with a recently proposed deterministic gradient sampling approach, this yields a deterministic and provably convergent way to approximate subdifferentials for computing descent directions.
}




\maketitle

\section{Introduction}

Nonsmooth optimization is concerned with the optimization of a function $f : \R^n \rightarrow \R$ which is not necessarily continuously differentiable. For such functions, one cannot rely on the gradient for describing the local behavior around a given point. As a replacement, generalized concepts from nonsmooth analysis can be employed. If the objective is still locally Lipschitz continuous, as is the case in many practical applications, then the standard approach is to use the \emph{Clarke subdifferential} $\partial f$ \cite{C1990}. However, since $\partial f(x)$ reduces to the gradient if $f$ is continuously differentiable at $x$, and since $f$ is typically continuously differentiable almost everywhere, the Clarke subdifferential cannot be used to capture nonsmoothness numerically. To circumvent this issue, the \emph{(Goldstein)} $\varepsilon$-\emph{subdifferential} $\partial_\varepsilon f$ \cite{G1977} may be used instead, which is the convex hull of all Clarke subdifferentials in an $\varepsilon$-ball around a given point. For the $\varepsilon$-subdifferential, it is sufficient for $x$ to have a distance of at most $\varepsilon$ to a nonsmooth point of $f$ to capture the nonsmoothness. As such, it can be interpreted as a stabilized version of the Clarke subdifferential. 

It is well-known that the element $\barv \in -\partial_\varepsilon f(x_0)$ with the smallest norm is a descent direction for $f$ at $x_0$ \cite{G1977,C1990}. Unfortunately, the full $\varepsilon$-subdifferential that is required to compute this direction is rarely available in practice and has to be approximated instead. To this end, in the \emph{gradient sampling method} \cite{BLO2002,BLO2005,BCL2020}, the idea is to approximate $\partial_\varepsilon f(x_0)$ by the convex hull of the gradients at randomly generated points in the $\varepsilon$-ball around $x_0$ where $f$ is differentiable. While this is easy to implement and convergence can be shown with probability $1$, randomly computing gradients means that one generally computes more gradients than would be necessary. 
Furthermore, a good approximation may require a large and a priori unknown number of sample points (which is highlighted in Appendix \ref{sec:appendix}).
As an alternative, in \cite{MY2012,GP2021,G2022}, a deterministic sampling approach was used. There, the idea is to compute the approximation of $\partial_\varepsilon f(x_0)$ iteratively, by starting with $W = \{ \xi \}$ for a subgradient $\xi \in \partial f(x_0)$ at $x_0$, and then adding new elements of $\partial_\varepsilon f(x_0)$ to $W$ until $\conv(W)$ is a satisfactory approximation. The mechanism for finding new elements of $\partial_\varepsilon f(x_0)$ is based on the observation that if $v$ is a direction that yields less descent than expected (based on the current approximation $\conv(W)$), then there has to be a point $x_0 + t v$ with $t > 0$ in the $\varepsilon$-ball at which a new subgradient $\xi' \in \partial f(x_0 + tv) \subseteq \partial_\varepsilon f(x_0)$ with $\xi' \notin \conv(W)$ can be sampled. To find such a $t$, a subroutine based on bisection of the interval $(0,\varepsilon/\| v \|)$ is used. While in \cite{MY2012,GP2021,G2022}, it was analyzed why this subroutine likely works (i.e., terminates) in practice and while termination was also observed in all numerical examples, a full proof (under reasonable assumptions for $f$) was not given.

The goal of this note is to close the above mentioned gap in the convergence theory of the deterministic gradient sampling approach. The bisection algorithm in \cite{MY2012,GP2021,G2022} is based on reformulating the problem of finding a new element of $\partial_\varepsilon f(x_0)$ as finding a point $t > 0$ in which a certain nonsmooth function $h : \R \rightarrow \R$ is increasing. The convergence issues arise in cases where the bisection converges to a critical point $\bart$ of $h$ (i.e., to a point $\bart$ with $0 \in \partial h(\bart)$). To fix these issues, we replace $h$ by a slightly modified function $\tilh$. We then show convergence of the resulting method for the case where $f$ is weakly lower semismooth \cite{M1977a,LO2013} (meaning that $-f$ is weakly upper semismooth). Since semismooth functions \cite{M1977b} are weakly lower semismooth, this case includes continuously differentiable functions, convex functions and piecewise differentiable functions \cite{SS2008}. As our result is essentially just concerned with the deterministic computation of new $\varepsilon$-subgradients (not necessarily in the context of \cite{MY2012,GP2021,G2022}), it has also use in other methods based on gradient sampling, like \cite{CQ2013,CQ2015}. Our method has strong similarities to Procedure 4.1 in \cite{K2009}, which is used for the computation of new subgradients in a bundle framework. 
It is based on the same idea, but differs in the condition used for bisection and the stopping criterion.

The remainder of this article is structured as follows: In Section \ref{sec:descent_directions_for_nonsmooth_functions} we introduce the basics of gradient sampling and the bisection algorithm from \cite{GP2021} (which is identical to the one in \cite{G2022} and almost identical to the one in \cite{MY2012}). In \cite{GP2021}, the bisection was only a small part of a larger algorithm, but we will introduce it here in a stand-alone way for convenience. In Section \ref{sec:improved_bisection_method} we construct the improved bisection algorithm and show its convergence when $f$ is weakly lower semismooth. In Section \ref{sec:Examples} we visualize the behavior of the improved bisection method in a simple example. Finally, in Section \ref{sec:conclusion}, we summarize our results and discuss possible directions for future research.

\section{Computing descent directions for nonsmooth functions} \label{sec:descent_directions_for_nonsmooth_functions}

In this section, we summarize the basic ideas of gradient sampling from \cite{BLO2002,BLO2005,BCL2020,MY2012,GP2021,G2022}. To this end, let $f : \R^n \rightarrow \R$ be locally Lipschitz continuous. The \emph{Clarke subdifferential} \cite{C1990} of $f$ at $x \in \R^n$ is given by
\begin{equation} \label{eq:def_clarke_subdiff}
    \begin{aligned}
        \partial f(x) := \conv &\left( \left\{ \xi \in \R^n : \exists (x^j)_j \in \R^n \setminus \Omega \text{ with } \lim_{j \rightarrow \infty} x^j = x \text{ and } \right. \right. \\
        &\left. \left. \lim_{j \rightarrow \infty} \nabla f(x^j) = \xi \right\} \right),
    \end{aligned}
\end{equation}
where $\Omega \subseteq \R^n$ is the set of points at which $f$ is not differentiable. (By Rademacher's Theorem \cite{C1990}, $\Omega$ has measure zero.) The elements of $\partial f(x)$ are called \emph{(Clarke) subgradients}. In theory, the Clarke subdifferential can be used similarly to the standard gradient, as it can be used in generalized versions of results like the mean value theorem, the chain rule and optimality conditions \cite{C1990}. In practice however, there are severe problems: $\partial f(x)$ only captures the nonsmoothness of $f$ if $x$ is a point where $f$ is not continuously differentiable, which is typically a null set. Furthermore, when $x \in \Omega$, there is no general way of computing the full subdifferential, i.e., all subgradients. Instead, a reasonable assumption is that we can only evaluate a single, arbitrary subgradient from $\partial f(x)$ at any $x \in \R^n$. Further explanations and examples for these issues can be found in \cite{L1989}.
To circumvent these problems, a more suitable object to use as a generalized derivative for constructing descent methods is the \emph{(Goldstein) $\varepsilon$-subdifferential} \cite{G1977}
\begin{align*}
    \partial_\varepsilon f(x) := \conv \left( \bigcup_{y \in B_\varepsilon(x)} \partial f(y) \right),
\end{align*}
where $\varepsilon \geq 0$, $B_\varepsilon(x) := \{ y \in \R^n : \| y - x \| \leq \varepsilon \}$ and $\| \cdot \|$ is the Euclidean norm. The elements of $\partial_\varepsilon f(x)$ are called $\varepsilon$\emph{-subgradients}. The $\varepsilon$-subdifferential can be interpreted as a ``stabilized'' Clarke subdifferential. 
In particular, it can be used to compute descent directions, as we demonstrate in the following.

Let $x_0 \in \R^n$, $v \in \R^n \setminus \{ 0 \}$ and $\varepsilon > 0$. Then a simple application of the mean value theorem (\cite{C1990}, Theorem 2.3.7) shows that
\begin{align} \label{eq:descent_via_mvt}
    f \left( x_0 + t v \right) \leq f(x_0) + t \max_{\xi \in \partial_\varepsilon f(x_0)} \langle \xi, v \rangle \quad \forall t \in \left( 0, \frac{\varepsilon}{\| v \|} \right].
\end{align}
Thus, directions $v$ with $\langle \xi, v \rangle < 0$ for all $\xi \in \partial_\varepsilon f(x_0)$ are descent directions for $f$ at $x_0$. Based on convex analysis \cite{CG1959}, 
the direction that minimizes the maximum of the inner products on the right-hand side of \eqref{eq:descent_via_mvt}, called the \emph{$\varepsilon$-steepest descent direction},
can be computed as
\begin{align} \label{eq:qop_bar_v}
    \barv := -\argmin_{\xi \in \partial_\varepsilon f(x_0)} \| \xi \|^2.
\end{align}
It holds either $\barv = 0$, in which case $0 \in \partial_\varepsilon f(x_0)$ and $x_0$ is called $\varepsilon$\emph{-critical}, or
\begin{align} \label{eq:bar_v_ineq}
    \langle \xi, \barv \rangle \leq - \| \barv \|^2 < 0 \quad \forall \xi \in \partial_\varepsilon f(x_0)
\end{align}
and, due to \eqref{eq:descent_via_mvt}, 
\begin{align} \label{eq:armijo_bar_v}
    f \left( x_0 + \frac{\varepsilon}{\| \barv \|} \barv \right) \leq f(x_0) - \varepsilon \| \barv \|.
\end{align}
Unfortunately, the full $\varepsilon$-subdifferential required to solve Problem \eqref{eq:qop_bar_v} is rarely available in practice. Thus, the direction $\barv$ has to be approximated.

\subsection{Random gradient sampling}

In the standard gradient sampling framework \cite{BLO2002,BLO2005,BCL2020}, the $\varepsilon$-subdifferential is approximated by randomly (independently and uniformly) sampling $m \geq n+1$ elements $x^1, \dots, x^m \in B_\varepsilon(x_0) \setminus \Omega$ and setting
\begin{align*}
    W := \{ \nabla f(x_0), \nabla f(x^1), \dots, \nabla f(x^m) \} \subseteq \partial_\varepsilon f(x_0).
\end{align*}
(The differentiability of $f$ at the current iterate $x_0$ is enforced via a differentiability check.) As an approximation of $\barv$ from \eqref{eq:qop_bar_v}, the direction
\begin{align} \label{eq:def_vGS}
    v^{\text{GS}} := -\argmin_{\xi \in \conv(W)} \| \xi \|^2
\end{align}
is computed, and an Armijo-like backtracking line search is used to assure decrease in $f$. It can be shown that when dynamically reducing the sampling radius $\varepsilon$, then the resulting descent method (cf. \cite{BLO2005}, GS Algorithm) produces a sequence converging to a critical point of $f$ with probability $1$. Unfortunately, sampling randomly means that a large number of sample points $m$ may be required to assure that $v^{\text{GS}}$ is a good approximation of $\barv$, i.e., to assure that meaningful decrease is achieved in every descent step. This drawback is highlighted in Appendix \ref{sec:appendix}. 

\subsection{Deterministic gradient sampling}
\label{subsec:deterministic_gradient_sampling}

Instead of randomly sampling points from $B_\varepsilon(x_0)$ for the approximation of $\partial_\varepsilon f(x_0)$, there is also a deterministic approach \cite{GP2021}. (In \cite{GP2021} \emph{multiobjective} problems are considered, but we will only consider the special case of a single objective here.) Assume that $x_0$ is not $\varepsilon$-critical. The idea is to start with a subset $W \subseteq \partial_\varepsilon f(x_0)$ (e.g., $W = \{ \xi \}$ for $\xi \in \partial f(x_0)$) and to then iteratively add new subgradients to $W$ until a direction 
\begin{align} \label{eq:def_tilv}
    \tilv := -\argmin_{\xi \in \conv(W)} \| \xi \|^2
\end{align}
is found that yields sufficient descent. The meaning of ``sufficient descent'' can be derived from \eqref{eq:armijo_bar_v}: For some fixed $c \in (0,1)$, we want to have at least
\begin{align} \label{eq:sufficient_descent}
    f \left( x_0 + \frac{\varepsilon}{\| \tilv \|} \tilv \right) \leq f(x_0) - c \varepsilon \| \tilv \|.
\end{align}
To find a new subgradient (that is not already contained in $\conv(W)$) in case $\tilv$ does not yield sufficient descent, note that the mean value theorem implies that there are $t' \in ( 0, \varepsilon / \| \tilv \| )$ and $\xi' \in \partial f(x_0 + t' \tilv)$ with
\begin{align} \label{eq:new_subgradient_ineq}
    &\frac{\varepsilon}{\| \tilv \|} \langle \xi', \tilv \rangle = f \left( x_0 + \frac{\varepsilon}{\| \tilv \|} \tilv \right) - f(x_0) > -c \varepsilon \| \tilv \| \notag \\
    \Leftrightarrow \quad &  \langle \xi', \tilv \rangle > -c \| \tilv \|^2.
\end{align}
Analogously to \eqref{eq:qop_bar_v} and \eqref{eq:bar_v_ineq}, it holds
\begin{align*}
    \langle \xi, \tilv \rangle \leq -\| \tilv \|^2 < 0 \quad \forall \xi \in \conv(W).
\end{align*}
Thus, \eqref{eq:new_subgradient_ineq} implies that $\xi' \notin \conv(W)$, so adding $\xi'$ to $W$ improves the approximation of $\partial_\varepsilon f(x_0)$. In \cite{MY2012,GP2021}, it was proven that iteratively computing $\tilv$ via \eqref{eq:def_tilv} while adding new subgradients to $W$ as above yields a finite algorithm which deterministically computes descent directions satisfying \eqref{eq:sufficient_descent} (as long as  $0 \notin \partial_\varepsilon f(x_0)$). (For a formal definition of this algorithm, see Algorithm 2 in \cite{GP2021} for $k = 1$.)

Unfortunately, the above application of the mean value theorem does not yield an explicit formula for the computation of $\xi'$ as in \eqref{eq:new_subgradient_ineq}, so additional effort is required in practice. To this end, a strategy based on bisection can be used: Let 
\begin{align} \label{eq:def_h}
    h : \R \rightarrow \R, \quad t \mapsto f(x_0 + t \tilv) - f(x_0) + c t \| \tilv \|^2.
\end{align}
By the chain rule (\cite{C1990}, Theorem 2.3.9) it holds
\begin{align} \label{eq:chain_rule_h}
    \partial h(t) \subseteq \langle \partial f(x_0 + t \tilv), \tilv \rangle + c \| \tilv \|^2,
\end{align}
so $\partial h(t') \cap \R^{>0} \neq \emptyset$ for $t' \in (0,\varepsilon / \| \tilv \|)$ would imply that there is some $\xi' \in \partial f(x_0 + t' \tilv)$ as in \eqref{eq:new_subgradient_ineq}. Thus, roughly speaking, the idea is to search for some interval in $(0,\varepsilon / \| \tilv \|)$ on which $h$ is monotonically increasing. In \cite{GP2021} this was done via Algorithm \ref{algo:old_bisection}.
\begin{algorithm}
	\caption{Computation of new $\varepsilon$-subgradients}
	\label{algo:old_bisection}
	\begin{algorithmic}[1] 
		\Require Point $x_0 \in \R^n$, radius $\varepsilon > 0$, descent parameter $c \in (0,1)$, direction $\tilv \in \R^n \setminus \{ 0 \}$ violating \eqref{eq:sufficient_descent}.
		\State Initialize $j = 1$, $a_1 = 0$, $b_1 = \frac{\varepsilon}{\| \tilv \|}$ and $t_1 = \frac{1}{2} (a_1 + b_1)$.
		\State Compute $\xi' \in \partial f(x_0 + t_j \tilv)$.
		\State If $\langle \xi', \tilv \rangle > -c \| \tilv \|^2$ then stop.
		\State If $h(b_j) > h(t_j)$ then set $a_{j+1} = t_j$ and $b_{j+1} = b_j$. Otherwise set $a_{j+1} = a_j$ and $b_{j+1} = t_j$. 
		\State Set $t_{j+1} = \frac{1}{2}(a_{j+1}+b_{j+1})$, $j = j+1$ and go to step 2.
	\end{algorithmic} 
\end{algorithm}
It performs bisections such that $h(a_j) < h(b_j)$ for all $j \in \N$, while checking whether a $\xi'$ was found satisfying \eqref{eq:new_subgradient_ineq}. In \cite{MY2012,GP2021,G2022}, it was argued why the algorithm is likely to terminate in practice, and termination was also observed in all numerical examples, but a proper analysis was not carried out. There are basically two issues that may cause the algorithm not to terminate:
\begin{enumerate}
    \item One may never encounter a $t_j$ with $\partial h(t_j) \cap \R^{\geq 0} \neq \emptyset$ during the bisection, even if $\partial h(\bart) \cap \R^{>0} \neq \emptyset$ holds for the limit $\bart \in [0,\varepsilon / \| \tilv \|]$ of $(t_j)_j$. 
    \item The subgradient $\xi'$ evaluated in step 2 may not correspond to a subgradient $\langle \xi', \tilv \rangle + c \| \tilv \|^2 \in \partial h(t_j)$, since we do not have equality in \eqref{eq:chain_rule_h}. So one could have $\langle \xi', \tilv \rangle \leq -c\| \tilv \|^2$ even when $\partial h(t_j) \subseteq \R^{>0}$.
\end{enumerate}
In \cite{G2022}, Example 4.3.4, an example was constructed for which Algorithm \ref{algo:old_bisection} does not terminate with a function $h$ that is not semismooth (cf. \cite{M1977b}). In the following, we will construct a more nuanced example that shows that the algorithm may also fail for semismooth functions. 

\begin{example} \label{example:counterexample}
    For $i \in \N \cup \{ 0 \}$ let
    \begin{align*}
        x_i^1 &:= 1 - 7 \cdot 2^{-i-3}, \ \varphi_i^1 := 1 - 9 \cdot 2^{-2i - 3}, \\
        x_i^2 &:= 1 - 5 \cdot 2^{-i-3}, \ \varphi_i^2 := 1 - 3 \cdot 2^{-2i - 4}.
    \end{align*}
    Then $x_i^1 < x_i^2 < x_{i+1}^1$ for all $i \in \N \cup \{ 0 \}$. We construct a function $\varphi : \R \rightarrow \R$ as follows: For $x < 0$ let $\varphi(x) := -\frac{1}{2} x$, for $x \geq 1$ let $\varphi(x) := 1$ and on $[0,1)$ let $\varphi$ be the piecewise linear function with $\varphi(0) = 0$ and 
    \begin{align*}
        \varphi(x_i^1) = \varphi_i^1 \quad \text{and} \quad \varphi(x_i^2) = \varphi_i^2 \quad \forall i \in \N \cup \{ 0 \}.
    \end{align*}
    The graph of $\varphi$ is shown in Figure \ref{fig:counterexample}(a). Figure \ref{fig:counterexample}(b) shows the gradient of $\varphi$ at points where $\varphi$ is differentiable, and a vertical line from smallest to the largest subgradient at points where $\varphi$ is not differentiable. 
    \begin{figure}
        \centering
        \parbox[b]{0.32\textwidth}{
            \centering 
            \includegraphics[width=0.32\textwidth]{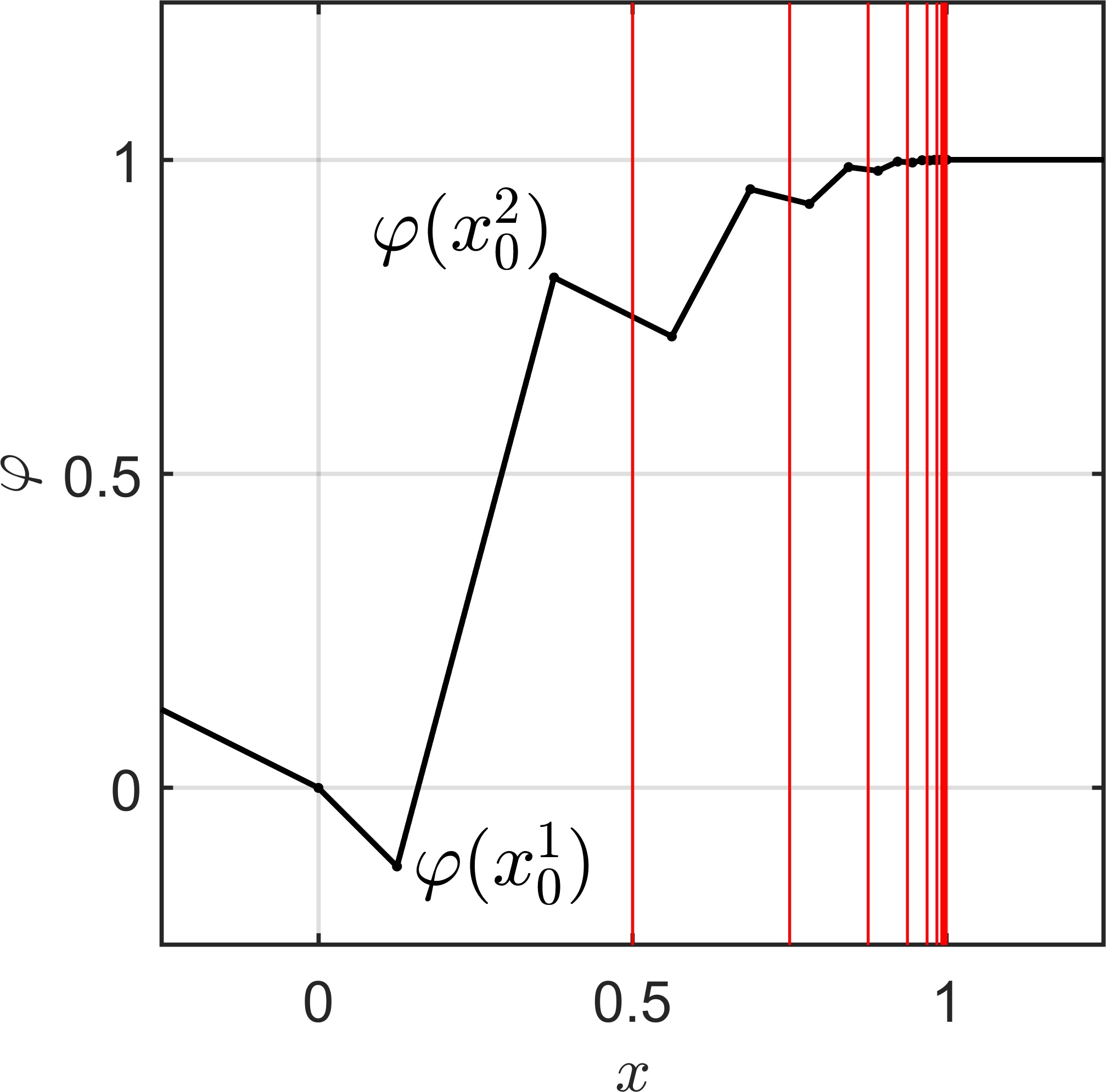}\\
            (a)
		}
        \parbox[b]{0.32\textwidth}{
			\centering 
			\includegraphics[width=0.32\textwidth]{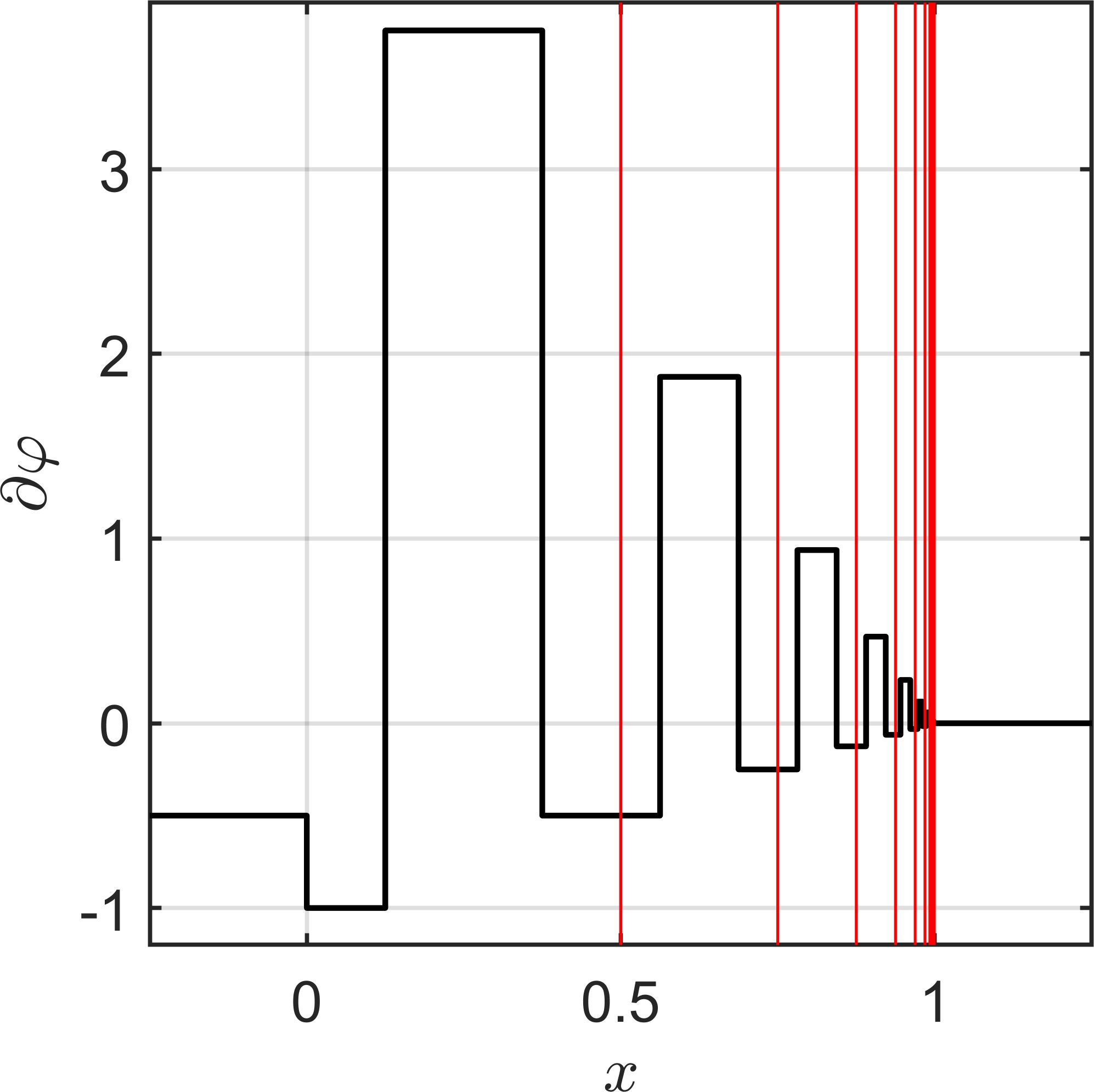}\\
            (b)
		}
        \parbox[b]{0.32\textwidth}{
			\centering 
			\includegraphics[width=0.32\textwidth]{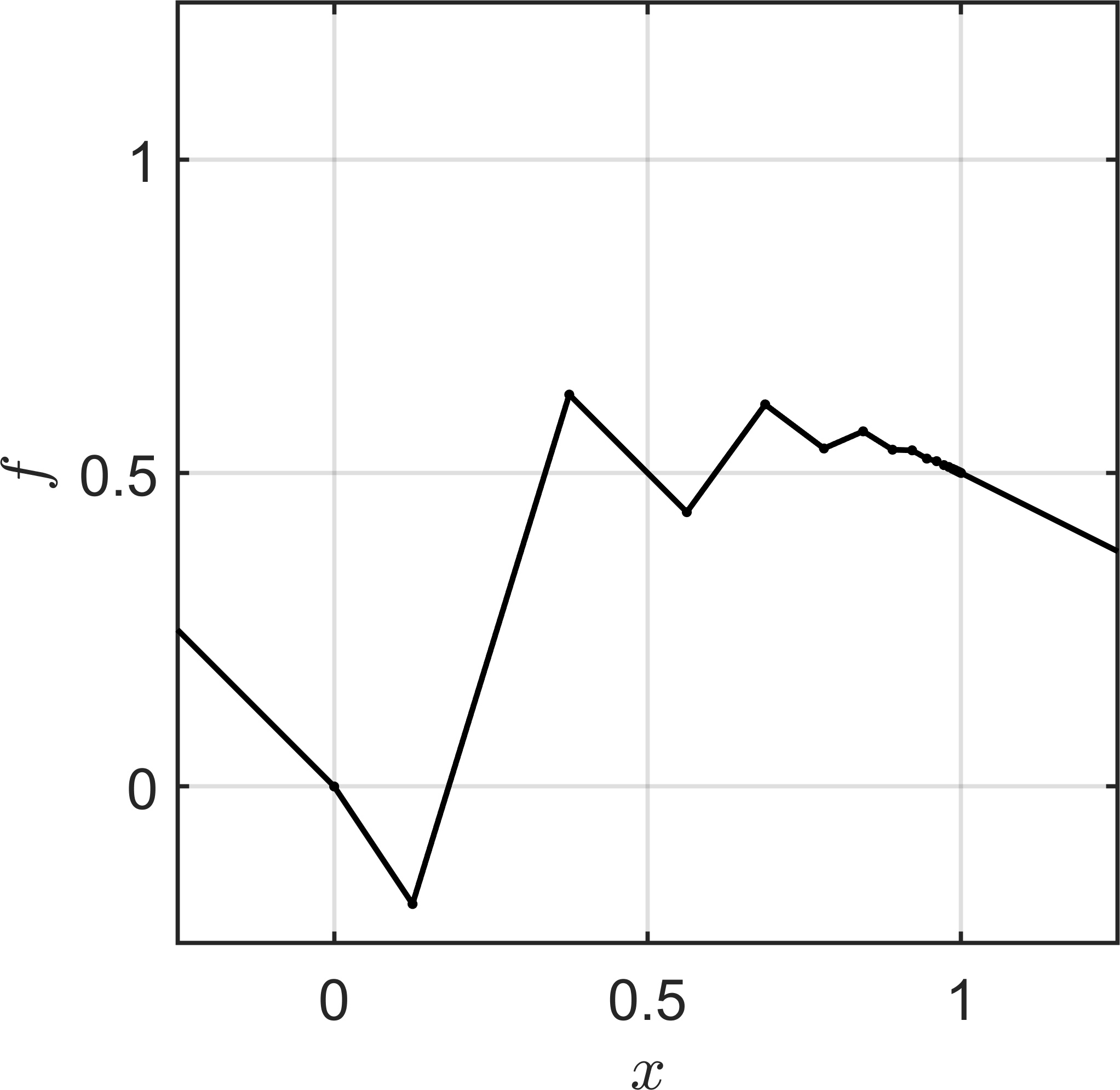}\\
            (c)
		}
        \caption{The graphs of (a) $\varphi = h$, (b) $\partial \varphi = \partial h$ and (c) $f$ in Example \ref{example:counterexample}. The red lines show the values of $t_j = 1 - 2^{-j}$ for $j \in \N$.}
        \label{fig:counterexample}
    \end{figure}
    We see that in the limit $x \rightarrow 1$, all subgradients tend to $0$. Based on this observation, it can be shown that $\varphi$ is semismooth. \\
    Now let 
    \begin{align*}
        f : \R \rightarrow \R, \quad x \mapsto \varphi(x) - \frac{1}{2} x
    \end{align*}
    as shown in Figure \ref{fig:counterexample}(c). Let $x_0 = 0$, $\varepsilon = 1$ and $c = 1/2$. Assume that we have evaluated the subgradient $\xi := -1 \in \partial f(0) = [-3/2,-1]$. Then for $W = \{ \xi \}$, the direction $\tilv$ from \eqref{eq:def_tilv} is simply $\tilv = -\xi = 1$. When checking whether $\tilv$ yields sufficient decrease, we see that
    \begin{align*}
        f \left( x_0 + \frac{\varepsilon}{\| \tilv \|} \tilv \right) - f(x_0) = f \left( 1 \right) - f(0) = \frac{1}{2} > -\frac{1}{2} = -c \varepsilon \| \tilv \|,
    \end{align*}
    i.e., $\tilv$ does not yield sufficient (or even any) decrease. \\
    For Algorithm \ref{algo:old_bisection} we have
    \begin{align*}
        h(t) = f(x_0 + t \tilv) - f(x_0) + c t \| \tilv \|^2 = f(t) + \frac{1}{2} t = \varphi(t).
    \end{align*}
    Since $1 = h(1) = h(b_1) > h(t)$ for all $t \in [0,1)$, we have
    \begin{align*}
        a_j = 1 - 2^{-j+1}, \ b_j = 1, \ t_j = 1 - 2^{-j} \quad \forall j \in \N,
    \end{align*}
    as indicated in Figure \ref{fig:counterexample}(a) and (b). By construction (cf. Figure \ref{fig:counterexample}(b)), $f$ is continuously differentiable at all $x_0 + t_j \tilv = t_j$ and
    \begin{align*}
        \partial f(x_0 + t_j \tilv) = \{ \nabla f(t_j) \} = \left\{ -2^{-j} - \frac{1}{2} \right\} \quad \forall j \in \N,
    \end{align*}
    so
    \begin{align*}
        \langle \xi', \tilv \rangle = -2^{-j} - \frac{1}{2} < -\frac{1}{2} = -c \| \tilv \|^2 \quad \forall j \in \N.
    \end{align*}
    Thus, Algorithm \ref{algo:old_bisection} does not terminate.
    
\end{example}

Note that in the previous example, only the nonsmoothness of $f$ at $x = 1$ is relevant for the failure of the algorithm. Thus, by ``smoothing'' the objective $f$ at all kinks except $1$, one could even construct a semismooth function which is continuously differentiable everywhere outside of a single point for which the algorithm fails.
    
\section{Improved bisection method} \label{sec:improved_bisection_method}

In this section, we propose a slightly modified version of Algorithm \ref{algo:old_bisection} and prove termination for the case where the objective $f$ is weakly lower semismooth. For any $x \in \R^n$, we assume that we are only able to evaluate a single, arbitrary subgradient of the locally Lipschitz continuous $f$ at $x$. 

\subsection{Derivation of the improved method}

Roughly speaking, the idea is to use a smaller parameter $\tilc < c$ in $h$ (cf. \eqref{eq:def_h}) to try to find a new subgradient that satisfies a stricter version of inequality \eqref{eq:new_subgradient_ineq}. This will solve the first of the issues described in Section \ref{subsec:deterministic_gradient_sampling}, as even points in which $\partial h$ is negative but close to zero then suffice to find subgradients that satisfy the weaker requirement with respect to $c$. 

More precisely, note that $\tilv$ is an unacceptable descent direction if and only if 
\begin{align*}
    &f \left( x_0 + \frac{\varepsilon}{\| \tilv \|} \tilv \right) > f(x_0) - c \varepsilon \| \tilv \| \\
    \Leftrightarrow \quad & \cmin := -\frac{f \left( x_0 + \frac{\varepsilon}{\| \tilv \|} \tilv \right) - f(x_0)}{\varepsilon \| \tilv \|} < c.
\end{align*}
Thus, if $\tilv$ is an unacceptable direction, then it is also unacceptable if we replace $c$ in \eqref{eq:sufficient_descent} by any $\tilc \in (\cmin,c) \neq \emptyset$. In other words, we could apply Algorithm \ref{algo:old_bisection} for any $\tilc \in (\cmin,c)$ and the method would still produce sequences with $\tilh(a_j) < \tilh(b_j)$ for all $j \in \N$, where, analogously to \eqref{eq:def_h},
\begin{align*}
    \tilh : \R \rightarrow \R, \quad t \mapsto f(x_0 + t \tilv) - f(x_0) + \tilc t \| \tilv \|^2.
\end{align*}
In step $3$, the method would check whether $\langle \xi', \tilv \rangle > - \tilc \| \tilv \|^2$. But since $\tilc < c$, this inequality is stricter than \eqref{eq:new_subgradient_ineq}. Thus, instead, we apply Algorithm \ref{algo:old_bisection} with only $h$ being replaced by $\tilh$ and the rest unchanged. In terms of the subdifferential of $\tilh$, this means that the method may stop as soon as 
\begin{align} \label{eq:subdiff_h_stop_lower_bound}
    \exists g \in \partial \tilh(t_j) : g = \langle \xi', \tilv \rangle + \tilc \| \tilv \|^2 > (\tilc - c) \|\tilv\|^2,
\end{align}
where $(\tilc - c) \|\tilv\|^2 < 0$. For clarity, this modified algorithm is denoted in Algorithm \ref{algo:new_bisection}.

\begin{algorithm}
	\caption{Improved computation of new $\varepsilon$-subgradients}
	\label{algo:new_bisection}
	\begin{algorithmic}[1] 
		\Require Point $x_0 \in \R^n$, radius $\varepsilon > 0$, descent parameters $c \in (0,1)$, $\tilc \in (\cmin,c)$, direction $\tilv \in \R^n \setminus \{ 0 \}$ violating \eqref{eq:sufficient_descent}.
		\State Initialize $j = 1$, $a_1 = 0$, $b_1 = \frac{\varepsilon}{\| \tilv \|}$ and $t_1 = \frac{1}{2} (a_1 + b_1)$.
		\State Compute $\xi' \in \partial f(x_0 + t_j \tilv)$.
		\State If $\langle \xi', \tilv \rangle > -c \| \tilv \|^2$ then stop.
		\State If $\tilh(b_j) > \tilh(t_j)$ then set $a_{j+1} = t_j$ and $b_{j+1} = b_j$. Otherwise set $a_{j+1} = a_j$ and $b_{j+1} = t_j$. 
		\State Set $t_{j+1} = \frac{1}{2}(a_{j+1}+b_{j+1})$, $j = j+1$ and go to step 2.
	\end{algorithmic} 
\end{algorithm}

\subsection{Proof of termination}

We begin the analysis of Algorithm \ref{algo:new_bisection} with some simple, technical results.
\begin{lemma} \label{lem:seq_properties}
    If Algorithm \ref{algo:new_bisection} does not terminate, then
    \begin{enumerate}
        \item[(i)] $(a_j)_j$, $(b_j)_j$ and $(t_j)_j$ have the same limit $\bart \in [0,\varepsilon / \| \tilv \|]$,
        \item[(ii)] $\bart \in [a_j,b_j]$ for all $j \in \N$,
        \item[(iii)] $(\tilh(b_j))_j$ is monotonically increasing and $\tilh(a_j) < \tilh(b_j)$ for all $j \in \N$,
        \item[(iv)] $t_j < \bart$ for infinitely many $j \in \N$.
    \end{enumerate}
\end{lemma}
\begin{proof}
    (i) By construction $(a_j)_j$ is monotonically increasing and $(b_j)_j$ is monotonically decreasing in $[0,\varepsilon / \| \tilv \| ]$, so both sequences converge. Since Algorithm \ref{algo:new_bisection} does not terminate, it holds $\lim_{j \rightarrow \infty} b_j - a_j = 0$, so they must converge to the same limit $\bart \in [0,\varepsilon / \| \tilv \| ]$. Since $t_j \in (a_j,b_j)$ for all $j \in \N$, $(t_j)_j$ must have the same limit. \\
    (ii) Assume that $\bart \notin [a_j,b_j]$ for some $j \in \N$. Then either $\bart < a_j$ or $\bart > b_j$. Due to the monotonicity of $(a_j)_j$ and $(b_j)_j$, this is a contradiction to $\bart$ being the limit of both sequences. \\
    (iii) By construction, the value of $(b_j)_j$ only changes when $\tilh(b_j) \leq \tilh(t_j)$ in step 4. In this case $b_{j+1}$ is set to $t_j$, so we have $\tilh(b_{j+1}) = \tilh(t_j) \geq \tilh(b_j)$ and $\tilh(a_{j+1}) = \tilh(a_j)$. If, on the other hand, $\tilh(b_j) > \tilh(t_j)$, then $a_{j+1} = t_j$, so $\tilh(a_{j+1}) = \tilh(t_j) < \tilh(b_j) = \tilh(b_{j+1})$. The proof follows by induction.  \\
    (iv) Assume that this does not hold. Then there is some $N \in \N$ such that $t_j \geq \bart$ for all $j > N$. By construction of $(t_j)_j$, $t_j = \bart$ may only hold once, so we can assume w.l.o.g. that $t_j > \bart$ for all $j > N$. Since $\bart \in [a_j,b_j]$ for all $j \in \N$ and $t_j$ is the midpoint of $[a_j,b_j]$, this implies that $a_j = a_N$ for all $j > N$. In particular, $\bart = \lim_{j \rightarrow \infty} a_j = a_N$. Since $\tilh(\bart) = \tilh(a_N) < \tilh(b_N)$ and $(\tilh(b_j))_j$ is monotonically increasing with $\lim_{j \rightarrow \infty} \tilh(b_j) = \tilh(\bart)$ due to continuity, this is a contradiction.
\end{proof}

To be able to fix the second of the two issues mentioned in Section \ref{subsec:deterministic_gradient_sampling}, we need a stronger assumption for $f$. An easy way to solve the issue would be to force equality in the chain rule \eqref{eq:chain_rule_h} by assuming that $f$ is \emph{regular} (cf. \cite{C1990}, Definition 2.3.4). While the class of regular functions includes convex functions, even simple nonconvex functions like $x \mapsto -|x|$ are not regular. As such, this assumption would heavily restrict the applicability of the method. Fortunately, we do not actually need equality in \eqref{eq:chain_rule_h} as we are only interested in the behavior of $\langle \partial f(x_0 + t \tilv), \tilv \rangle$ for $t \rightarrow \bart$, and not necessarily in $\langle \partial f(x_0 + \bart \tilv), \tilv \rangle$ itself. 
Thus, we will see that it suffices to assume that $f$ is \emph{weakly lower semismooth} \cite{M1977a,LO2013}, which means that it is locally Lipschitz and for $x \in \R^n$, $v \in \R^n$ and sequences $(s_i)_i \in \R^{>0}$, $(\xi_i)_i \in \R^n$ with $s_i \searrow 0$ and $\xi_i \in \partial f(x + s_i v)$ for all $i \in \N$, it holds
\begin{align} \label{eq:def_f_wl_semismooth}
    \limsup_{i \rightarrow \infty} \ \langle \xi_i, v \rangle \leq \liminf_{s \searrow 0} \frac{f(x + sv) - f(x)}{s}.
\end{align}
Roughly speaking, weak lower semismoothness means that there is a semicontinuous relationship between directional derivatives and sequences of subgradients (via the inner product). In our case, we are interested in inequality \eqref{eq:def_f_wl_semismooth} for $x = x_0 + \bart \tilv$ and $v = -\tilv$. This will give us a lower estimate for $\langle \xi',\tilv \rangle$ in step 3 of Algorithm \ref{algo:new_bisection}, which we can use to show termination.

To this end, we will first derive an upper bound for the right-hand side of \eqref{eq:def_f_wl_semismooth} in the following lemma.
\begin{lemma} \label{lem:f_wl_semismooth_prop}
    Assume that Algorithm \ref{algo:new_bisection} does not terminate. Let $\bart$ as in Lemma \ref{lem:seq_properties}. Then 
    \begin{align} \label{eq:f_wl_semismooth_prop}
        \liminf_{s \searrow 0} \frac{f(x_0 + \bart \tilv - s \tilv) - f(x_0 + \bart \tilv)}{s} \leq \tilc \| \tilv \|^2.
    \end{align}
\end{lemma}
\begin{proof}
    By Lemma \ref{lem:seq_properties} we have $t_j < \bart$ for infinitely many $j \in \N$. Let $(j_i)_i$ be the sequence of such $j$. Note that monotonicity of $(\tilh(b_j))_j$ and $t_{j_i} < \bart$ imply that
    \begin{align*}
        \tilh(\bart) = \lim_{j \rightarrow \infty} \tilh(b_j) \geq \tilh(b_{j_i}) > \tilh(t_{j_i}) \quad \forall i \in \N.
    \end{align*}
    In particular, writing $f(x_0 + t_{j_i} \tilv) = f(x_0 + \bart \tilv - (\bart - t_{j_i}) \tilv)$, it holds
    \begin{align*}
        & 0 > \tilh(t_{j_i}) - \tilh(\bart) = f(x_0 + \bart \tilv - (\bart - t_{j_i}) \tilv) - f(x_0 + \bart \tilv) + \tilc (t_{j_i} - \bart) \| \tilv \|^2 \\
        \Leftrightarrow \quad& \frac{f(x_0 + \bart \tilv - (\bart - t_{j_i}) \tilv) - f(x_0 + \bart \tilv)}{\bart - t_{j_i}} < \tilc \| \tilv \|^2
    \end{align*}
    for all $i \in \N$. Since $\bart - t_{j_i} \searrow 0$ and the limit inferior is taken in \eqref{eq:f_wl_semismooth_prop}, this completes the proof.
\end{proof}%
The previous lemma enables us to prove our main result.
\begin{theorem}%
    Assume that $f$ is weakly lower semismooth. Then Algorithm \ref{algo:new_bisection} terminates.%
\end{theorem}%
\begin{proof}%
    Assume that Algorithm \ref{algo:new_bisection} does not terminate. Choose $(j_i)_i$ as in the proof of Lemma \ref{lem:f_wl_semismooth_prop} and let $\xi'_i \in \partial f(x_0 + t_{j_i} \tilv)$ be the subgradient evaluated in step 3 in iteration $j_i$. Let $s_i := \bart -  t_{j_i}$. Then $\xi'_i \in \partial f(x_0 + t_{j_i} \tilv) = \partial f(x_0 + \bart \tilv - s_i \tilv)$ and $s_i \searrow 0$, so by Lemma \ref{lem:f_wl_semismooth_prop} and weak lower semismoothness, it holds 
    \begin{align*}
        \liminf_{i \rightarrow \infty}\ \langle \xi'_i, \tilv \rangle 
        &= -\limsup_{i \rightarrow \infty}\ \langle \xi'_i, -\tilv \rangle
        \geq -\liminf_{s \searrow 0} \frac{f(x_0 + \bart \tilv - s \tilv) - f(x_0 + \bart \tilv)}{s} \\
        &\geq -\tilc \| \tilv \|^2 > -c\| \tilv \|^2,
    \end{align*}
    since $\tilc \in (\cmin,c)$. In particular, we must have 
    \begin{align*}
         \langle \xi'_i, \tilv \rangle > -c\| \tilv \|^2
    \end{align*}
    after finitely many iterations, causing the algorithm to stop in step $3$.
\end{proof}%

\section{Example} \label{sec:Examples}

In this section, we visualize the difference between Algorithms \ref{algo:old_bisection} and \ref{algo:new_bisection} by revisiting Example \ref{example:counterexample}.

\begin{example} \label{example:fixed_counterexample}
    \begin{enumerate}
        \item[(a)] In the situation of Example \ref{example:counterexample}, it holds
        \begin{align*}
            \cmin = -\frac{f(0 + 1) - f(0)}{1} = -\frac{1}{2},
        \end{align*}
        so any $\tilc \in (-1/2,c) = (-1/2,1/2)$ can be chosen for Algorithm \ref{algo:new_bisection}. Figure \ref{fig:fixed_counterexample_1} shows the graphs of $\tilh$ and $\partial \tilh$ when choosing $\tilc = 1/4$.
        \begin{figure}
            \centering
            \parbox[b]{0.49\textwidth}{
                \centering 
                \includegraphics[width=0.325\textwidth]{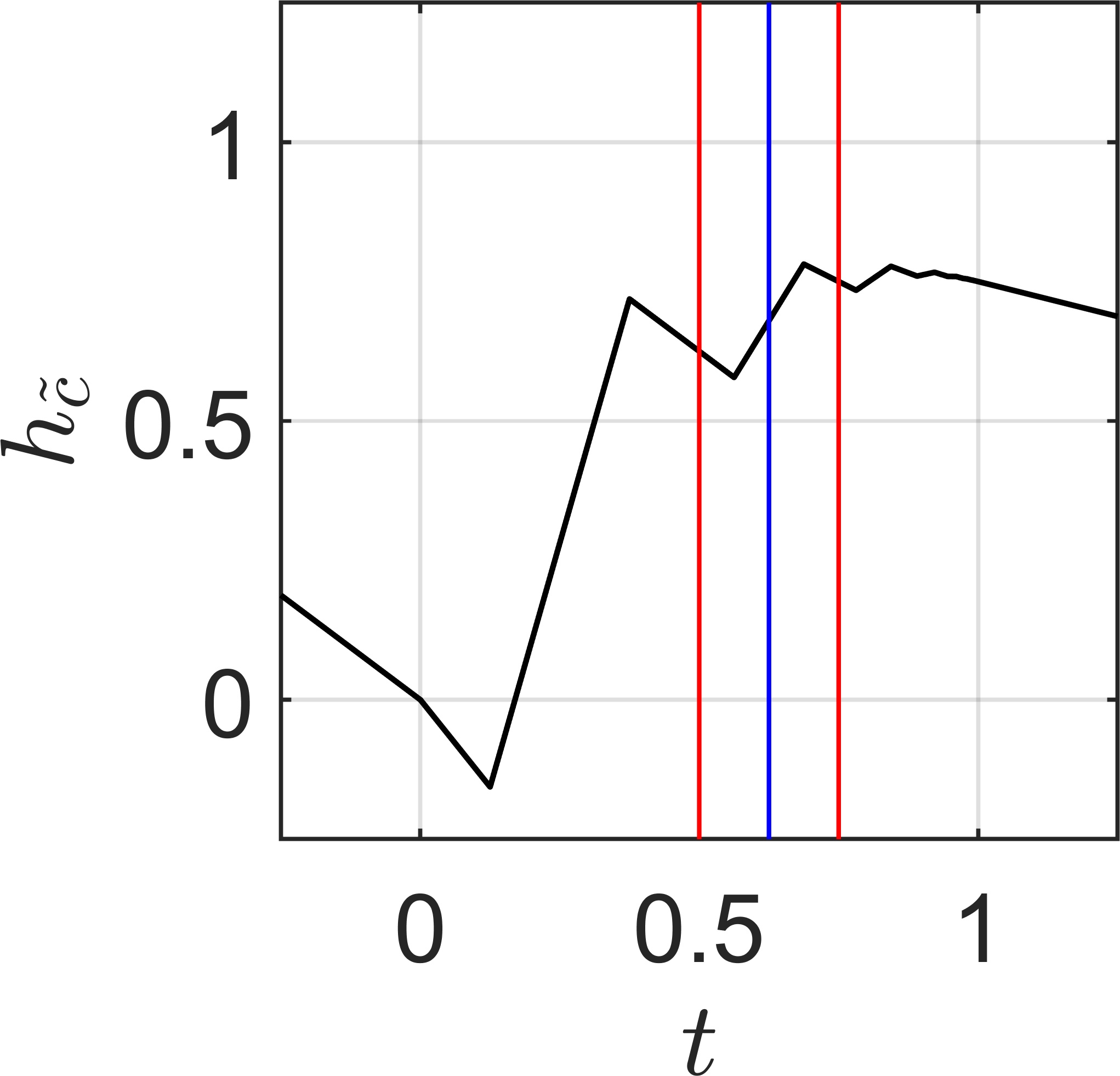}\\
                (a)
		      }
            \parbox[b]{0.49\textwidth}{
                \centering 
                \includegraphics[width=0.325\textwidth]{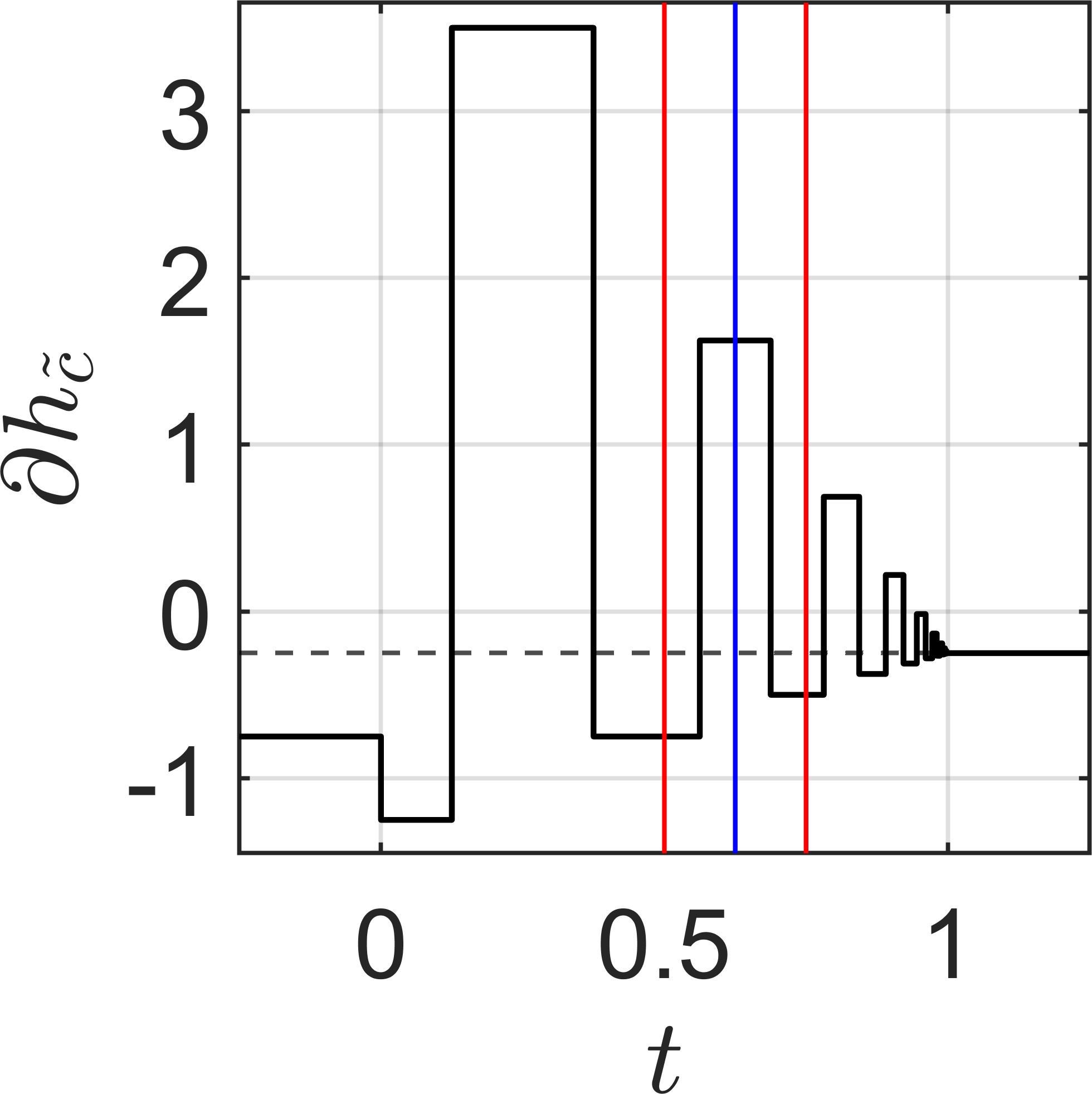}\\
                (b)
		      }
            \caption{The graphs of (a) $\tilh$ and (b) $\partial \tilh$ for $c = 1/2$, $\tilc = 1/4$ in Example \ref{example:fixed_counterexample}. The vertical lines show the sequence $(t_j)_j$ from Algorithm \ref{algo:new_bisection}, with the final value colored in blue. The dashed horizontal line in (b) marks the value $(\tilc - c) \| \tilv \|^2 = -1/4$ above which $\partial \tilh(t_j)$ must lie for the method to stop (cf. \eqref{eq:subdiff_h_stop_lower_bound}).}
            \label{fig:fixed_counterexample_1}
        \end{figure}
        We see that the algorithm terminates after two iterations with $t_3 = 5/8$ and the new $\varepsilon$-subgradient $\xi' = 11/8 \notin \conv(W) = \{ -1 \}$.
        \item[(b)] Note that part (a) essentially solved the problem by simply choosing a different value for $c$ in $h$ (leading to a more well-behaved function $\tilh$). To better visualize differences of Algorithms \ref{algo:old_bisection} and \ref{algo:new_bisection}, assume that we chose $c = 3/4$, such that we may choose $\tilc = 1/2 \in (\cmin,c)$. Then Algorithm \ref{algo:new_bisection} has to deal with the same problematic function as Algorithm \ref{algo:old_bisection} in Example \ref{example:counterexample}. The resulting graphs of $\tilh$ and $\partial \tilh$ are shown in Figure \ref{fig:fixed_counterexample_2}.
        \begin{figure}
            \centering
            \parbox[b]{0.49\textwidth}{
                \centering 
                \includegraphics[width=0.325\textwidth]{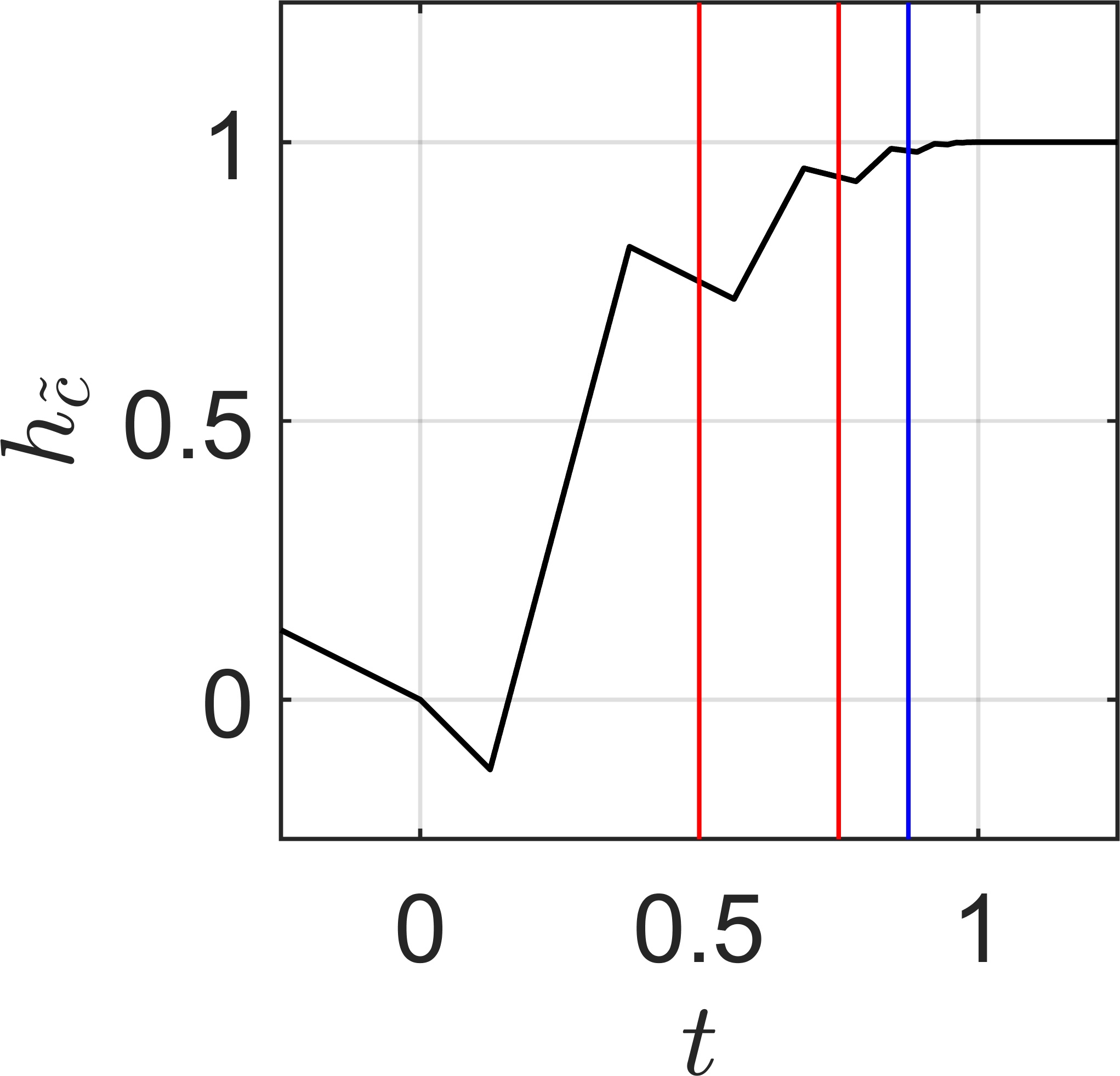}\\
                (a)
		      }
            \parbox[b]{0.49\textwidth}{
                \centering 
                \includegraphics[width=0.325\textwidth]{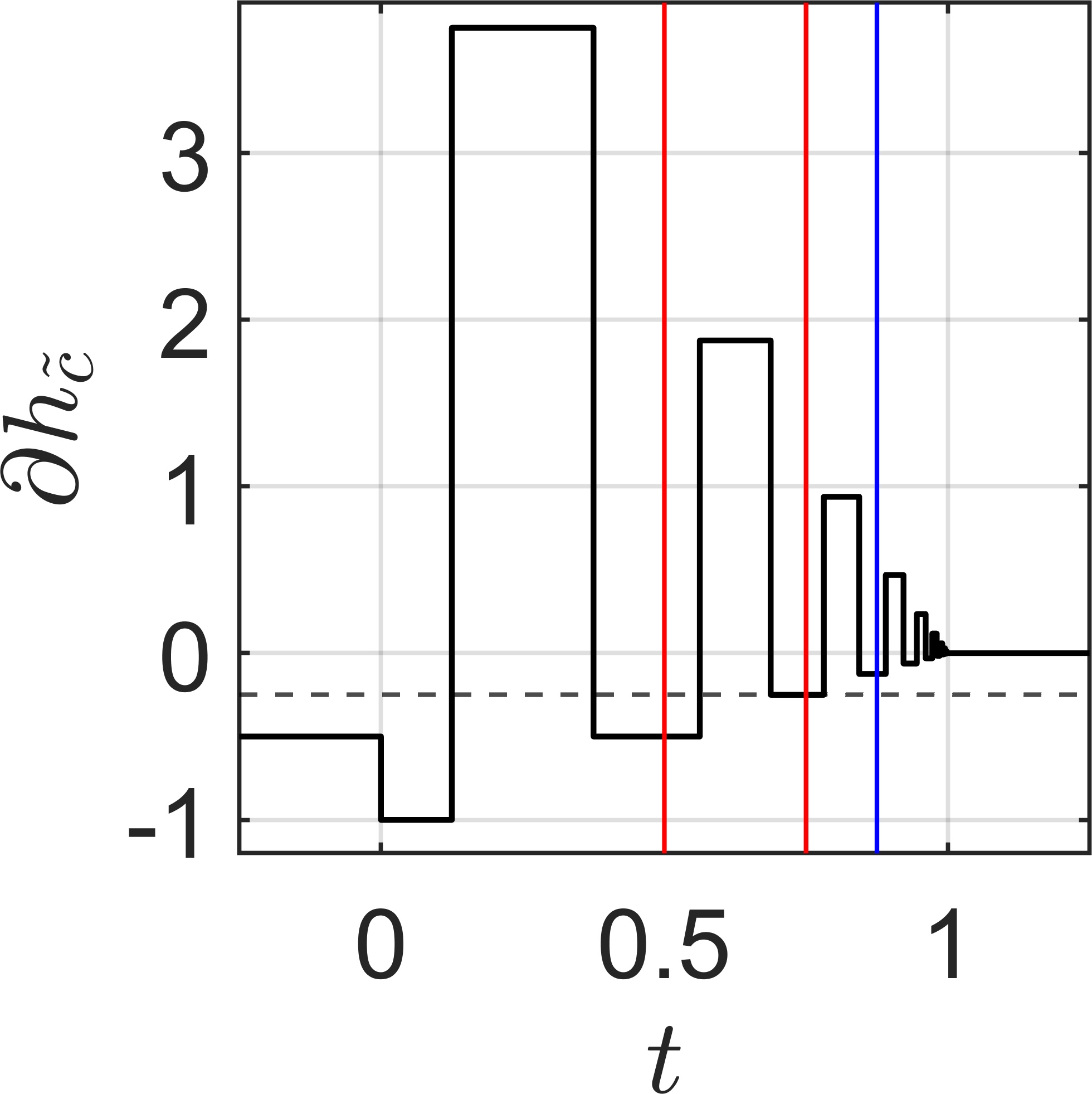}\\
                (b)
		      }
            \caption{Same as Figure \ref{fig:fixed_counterexample_1} but for $\tilc = 1/2$ and $c = 3/4$.}
            \label{fig:fixed_counterexample_2}
        \end{figure}
        Since, in Algorithm \ref{algo:new_bisection}, it is sufficient to have $g > (\tilc - c) \| \tilv \|^2$ for a subgradient $g = \langle \xi', \tilv \rangle + \tilc \| \tilv \|^2 \in \partial \tilh(t_j)$ (cf. \eqref{eq:subdiff_h_stop_lower_bound}), and since all subgradients at $\tilh(t)$ tend to $0$ as $t \rightarrow 1$, the method already stops in $t_3 = 7/8$ with the new $\varepsilon$-subgradient $\xi' = -5/8 \notin \conv(W) = \{ -1 \}$.
    \end{enumerate}
\end{example}

\section{Conclusion and future work} \label{sec:conclusion}

In this article, we showed how the gap in the convergence theory of the deterministic gradient sampling methods from \cite{MY2012,GP2021,G2022} can be closed for weakly lower semismooth functions by a more careful handling of the sufficient decrease condition in the bisection.

For future work, it might be worth to analyze the behavior of Algorithm \ref{algo:new_bisection} in a more general setting. In \cite{BL2021}, the convergence theory for the original gradient sampling method from \cite{BLO2005} was generalized to directionally Lipschitzian functions, and already in \cite{BLO2005}, gradient sampling was successfully applied to even more general, non-Lipschitzian functions.
Furthermore, the general strategy of approximating the $\varepsilon$-subdifferential in a deterministic fashion as in \cite{K2009,MY2012,GP2021,G2022} may lead to interesting new methods when combined with other methods that rely on random sampling, like \cite{CQ2013,CQ2015}. In the long run, we believe that this strategy may lead one step closer to a unified framework for both gradient sampling and bundle methods. For example, when comparing the standard gradient sampling method in \cite{BLO2005} and the bundle method of \cite{K2009}, one could ``hide'' all the null steps in the bundle method in a subroutine and end up with a method that is similar to the gradient sampling method, just with a different way to approximate the $\varepsilon$-subdifferential.

\bmhead{Data Availability}
Data sharing not applicable to this article as no datasets were generated or analyzed during
the current study.

\backmatter

\section*{Declarations}

\bmhead{Conflict of Interest} 
The author has no competing interests to declare that are relevant to the content of
this article.

\begin{appendices}

\section{A drawback of random sampling}\label{sec:appendix}

\begin{example} \label{example:GS_bad}
    For $n > 1$ let $\pr : \R^n \rightarrow \R^{n-1}, x \mapsto (x_1,\dots,x_{n-1})^\top$, be the projection onto the first $n-1$ components. Let
    \begin{align*}
        f : \R^n \rightarrow \R, \quad x \mapsto | x_n - \| \pr(x) \| | + \frac{1}{2} x_n.
    \end{align*}
    Then $f$ is continuously differentiable in $D^1 \cup D^2$, where
    \begin{align*}
        D^1 := \{ x \in \R^n : \pr(x) \neq 0, x_n < \| \pr(x) \| \}, \\
        D^2 := \{ x \in \R^n : \pr(x) \neq 0, x_n > \| \pr(x) \| \}.
    \end{align*}
    For $n = 2$, the graph of $f$ and the sets $D^1$, $D^2$ are shown in Figure \ref{fig:GS_bad}.
    \begin{figure}
        \centering
        \parbox[b]{0.49\textwidth}{
            \centering 
            \includegraphics[width=0.4\textwidth]{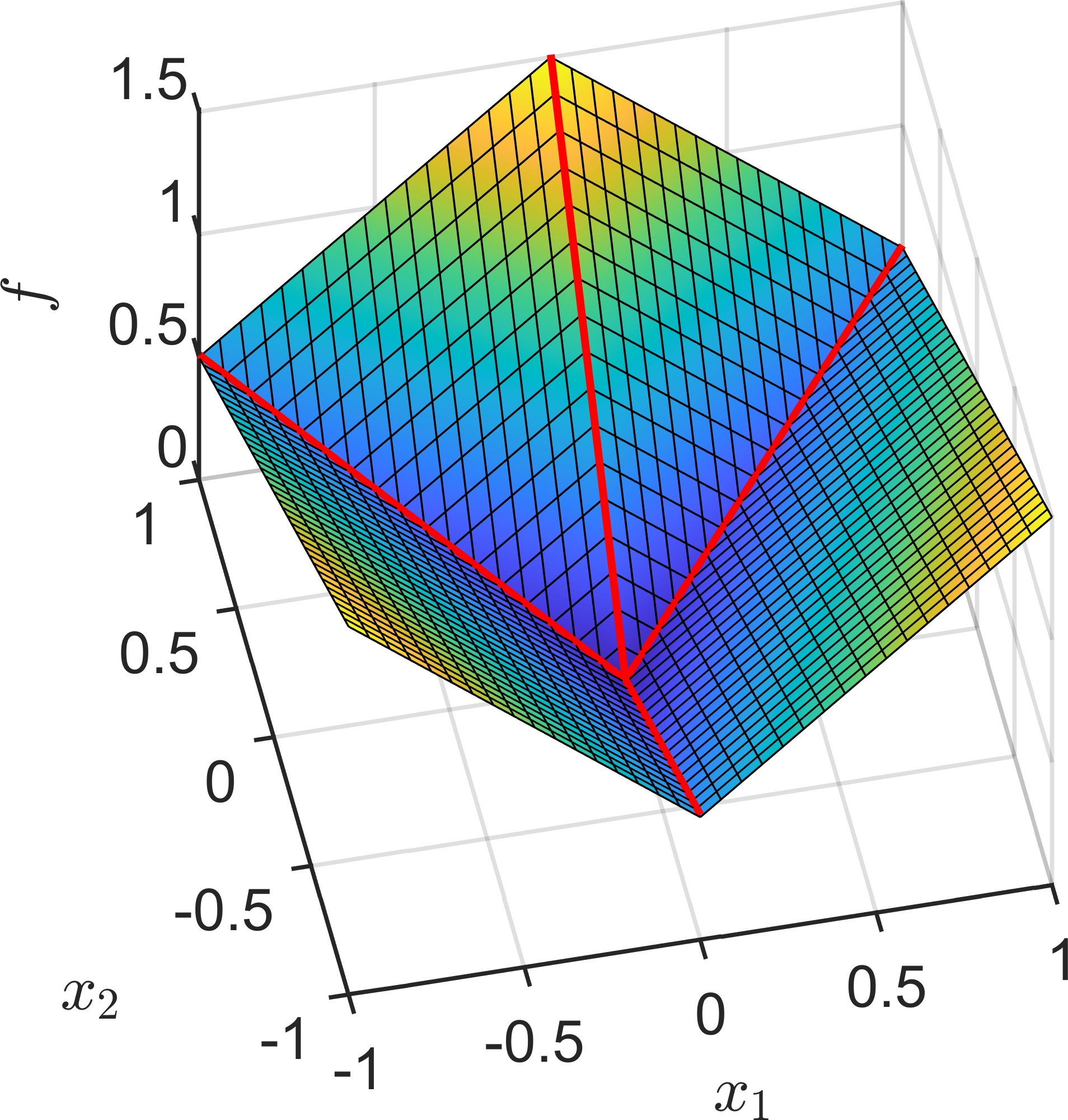}\\
            (a)
		}
        \parbox[b]{0.49\textwidth}{
            \centering 
            \includegraphics[width=0.4\textwidth]{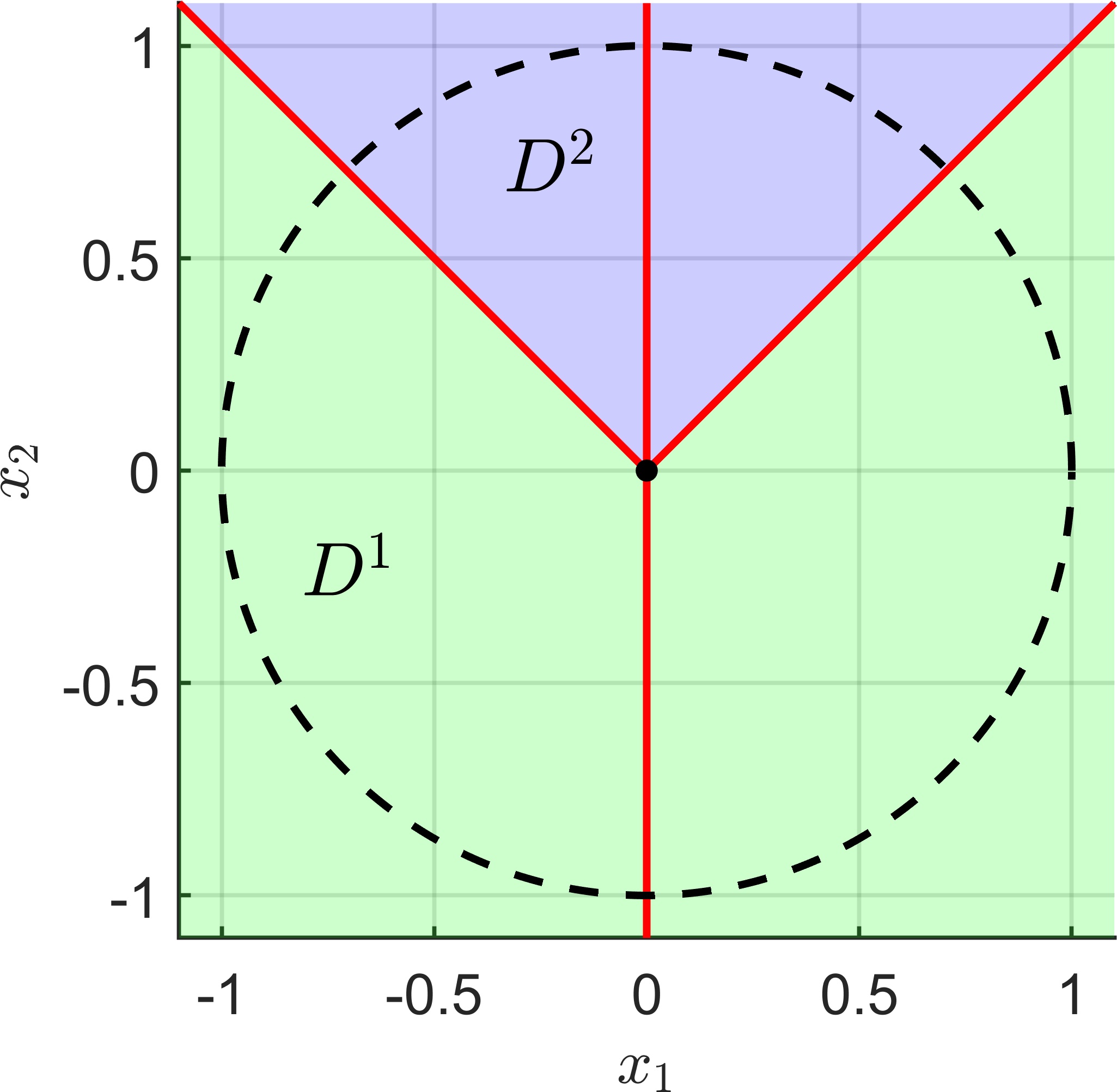}\\
            (b)
		}
        \caption{(a) The graph of $f$ for $n = 2$ in Example \ref{example:GS_bad}. The red lines indicate the points in which $f$ is not differentiable. (b) The boundary of the unit sphere $B_1(0)$ (dashed) and the sets $D^1$ (green) and $D^2$ (blue).}
        \label{fig:GS_bad}
    \end{figure}    
    For the gradient we have
    \begin{align*}
        \nabla f(x) = 
        \begin{pmatrix}
            \pr(x) / \| \pr(x) \| \\
            -1 / 2
        \end{pmatrix}
        \ \forall x \in D^1, \
        \nabla f(x) = 
        \begin{pmatrix}
            -\pr(x) / \| \pr(x) \| \\
            3 / 2
        \end{pmatrix}
        \ \forall x \in D^2.
    \end{align*}
    Thus, it is easy to see that $x = 0$ is a critical point of $f$. In particular, if $x_0 \in \R^n$ and $\varepsilon > 0$ such that $0 \in B_\varepsilon (x_0)$, then the solution of \eqref{eq:qop_bar_v} is $\barv = 0$. \\
    Clearly, to have $v^{\text{GS}} = 0$ in \eqref{eq:def_vGS}, it is necessary to sample at least one gradient from $D_\varepsilon^2 := D^2 \cap B_\varepsilon(x_0)$. For $x_0 = 0$ and $\varepsilon = 1$, the probability that $y \in D_\varepsilon^2$ for a uniformly sampled $y \in B_\varepsilon(x_0)$ can be computed by comparing the hypervolume $V_n(D_\varepsilon^2)$ of $D_\varepsilon^2$ (which can be computed via a partition of $D_\varepsilon^2$ into a hypercone and a hyperspherical cap) to $V_n(B_\varepsilon(x_0))$. In \cite{BLO2005}, $m = 2n$ gradients are sampled in every iteration for the approximation of $\partial_\varepsilon f(x_0)$. The resulting probabilities of having at least one of the $2n$ sample points in $D_\varepsilon^2$ are shown in Table \ref{tab:GS_probs}.
    \begin{table}[h]
        \centering
        \caption{Probability that at least one of the $2n$ sample points lies in $D_\varepsilon^2$ in Example \ref{example:GS_bad}.}
        \begin{tabular}{| c || c | c | c | c | c | c | c|}
            \hline
            $n$   & $2$      & $3$      & $5$      & $10$     & $20$     & $50$                & $100$ \\ 
            \hline
            &&&&&&& \\[-8pt] 
            Prob. & $0.6836$ & $0.6133$ & $0.4502$ & $0.1394$ & $0.0067$ & $3.3 \cdot 10^{-7}$ & $2.2 \cdot 10^{-14}$ \\
            \hline
        \end{tabular}
        \label{tab:GS_probs}
    \end{table}
    We see that when increasing the dimension $n$, it quickly becomes highly unlikely that random gradient sampling correctly identifies $x_0 = 0$ as $\varepsilon$-critical. Note that this is not related to $x_0 = 0$ being a nonsmooth point of $f$, and we get a similar (or even worse) result when choosing some $x_0 \in D^1$ close to zero. (It would just become more difficult to compute the exact probabilities as in the table for $x_0 \neq 0$.) In this case, the method from \cite{BLO2005} would perform descent steps with short step lengths (and little descent), which would require many function evaluations due to the backtracking nature of the line search, without recognizing that the iterates are already $\varepsilon$-critical. \\
    If deterministic sampling is used instead (cf. Section \ref{subsec:deterministic_gradient_sampling}), then a simple computation shows that if the first two subgradients $\xi^1$ and $\xi^2$ were sampled from $D^1$, then the next $\tilv$ from \eqref{eq:def_tilv} (for $W = \{ \xi^1, \xi^2 \}$) must be $\tilv = (0,\dots,0,1/2)^\top$, so the next subgradient is sampled at $x_0 + 1/2 \ \tilv$ (cf. Algorithm \ref{algo:new_bisection}). If $x_0$ is close to zero, then $x_0 + 1/2 \ \tilv \in D^2$ such that a subgradient from $D^2$ is sampled. After at most one additional iteration, the procedure stops and correctly obtains $\tilv = \barv = 0$. Thus, for arbitrary $n$, at most $4$ subgradients have to be sampled when sampling deterministically.
\end{example}

\end{appendices}

\bibliography{references}

\end{document}